\documentclass[11pt,reqno]{article}
\usepackage{amssymb, amsmath, amsthm, amsfonts, amscd, epsfig}
\usepackage{color}

\usepackage[english]{babel}
\usepackage{bm}
\usepackage{bbm}
\usepackage{graphicx, epsfig}
\usepackage{geometry,graphicx,pgfplots}
\usepackage{tikz}

\usepackage{subcaption}
\usepackage[export]{adjustbox}
\usepackage[titletoc,toc,title]{appendix}
\usepackage[normalem]{ulem} 

\newtheorem{theorem}{Theorem}[section]
\newtheorem{cor}[theorem]{Corollary}
\newtheorem{lemma}[theorem]{Lemma}                                                                                                                                                                                                                                                                             

\newtheorem{definition}{Definition}

\newtheorem{remark}{Remark}

\def\ep{\epsilon}

\newcommand{\Bp}{\mathbf{p}}

\newcommand{\Bc}{\mathbf{c}}

\newcommand{\Om}{\Omega}
\newcommand{\Bs}{\mathbf{s}}
\newcommand{\Bx}{\mathbf{x}}

\newcommand{\RR}{\mathbb{R}}

\newcommand{\NN}{\mathbb{N}}

\newcommand{\Scal}{\mathcal{S}}

\newcommand{\quadand}{\quad\mbox{and}\quad}
\newcommand{\quadwith}{\quad\mbox{with}\quad}
\newcommand{\p}{\partial}

\newcommand{\pd}[2]{\frac {\p #1}{\p #2}}
\newcommand{\ds}{\displaystyle}
\newcommand{\eqnref}[1]{(\ref {#1})}

\newcommand{\beq}{\begin{equation}}
\newcommand{\eeq}{\end{equation}}
\newcommand{\RN}[1]{%
  \textup{\uppercase\expandafter{\romannumeral#1}}%
}

\numberwithin{equation}{section}
\numberwithin{figure}{section}

\begin{document}

\newcommand{\TheTitle}{Characterization of stress concentration in two-dimensional boundary value problems: Neumann-type and Dirichlet-type}

\newcommand{\TheAuthors}{J. Hong and M. Lim}

\title{{\TheTitle}\thanks{{This work is supported by the Korean Ministry of Science, ICT and Future Planning through NRF grant No. 2016R1A2B4014530.}}}
\author{
Jiho Hong\thanks{\footnotesize Department of Mathematical Sciences, Korea Advanced Institute of Science and Technology, Daejeon 34141, Korea ({jihohong@kaist.ac.kr}, {mklim@kaist.ac.kr}).}\footnotemark[2] \and Mikyoung Lim\footnotemark[2]}

\maketitle
%
%
%
%
%
%
%

\begin{abstract}
We consider a boundary value problem for the conductivity equation in a bounded domain containing an inclusion which is nearly touching to the domain's boundary. We assume that the domain and the inclusion are disks with conductivity jump on the boundary of the inclusion. By using the layer potential technique and adopting the bipolar coordinates, we derive the asymptotic formulas which explicitly describe the gradient blow-up of the solution as the distance between the inclusion and the domain's boundary tends to zero. It turns out that the gradient blow-up term can be identified with the electric field generated by certain kind of virtual line charges supported on line segments outside of the domain; thereby, the gradient blow-up is completely characterized in terms of both of Neumann-type and Dirichlet-type boundary conditions, conductivities and geometric parameters.

%
\end{abstract}



%
\section{Introduction}

The stress concentration in composite materials has gotten attention from the researchers in many fields of interest \cite{KLY:2014:CEFCT, BLY:2010:GEPICP, RS09, AKLLL:2007:OEEFT, LV:2000:GESDF} due to its practical applications. For example \cite{G:1984:TEPL}, if an air gap is subjected to a high voltage above a threshold, a plasmonic phenomenon called ``corona discharge'' takes place. The corona discharge has negative effects on the electric power transmission: power loss, ozone production, noise and damage to insulation. On the other hand, the ionization of fluid due to the corona discharge is applied to useful corona devices: electrostatic precipitator, photocopier, nitrogen lazer and ionic wind devices. Hence, starting from \cite[Morrow]{M:1997:TPGC}, there has been attempts \cite{KP:2004:SCDPPC, F:1999:AGFEM} to numerically simulate the phenomenon.
%

Unfortunately, when a composite material has eccentric geometry and steep jumps in material properties, existing numerical solvers for the corresponding problem demand a high cost. Instead, mathematical characterization of the stress concentration can provide intuition about controlling the phenomenon to achieve practical values. As for the electric field concentration, mathematicians have found explicit formulas \cite{AKL,LYu3D,LY:2015:ASCE} that approximate the solutions based on the theory of layer potentials. Recently, a study of anti-plane elasticity \cite{imageCharge2018} classified the blow-up behavior of stress in the free space $\mathbb{R}^2$ under core-shell structure.

Another approach using variational principles has been partially successful in linear elasticity. The upper bound of the increasing rate of stress was obtained by Bao et al, as the distance $\ep$ between two convex inclusions gets smaller. The rate turned out to be $O(\ep^{-1/2})$ in the planar case \cite{bao2015gradient}, and $O(|\ep\log\ep|^{-1})$ in three dimensions \cite{bao2017gradient}, respectively. Bao et al \cite{bao2017optimal} also estimated the stress concentration in all dimensions with the core-shell type geometry of elastic materials. However, so far, there has been no research about finding explicit asymptotic formula for accumulated stress in Lam{\'e} systems with core-shell geometry under generally given incident field.

In this article, we investigate the gradient blow-up of the solution to the two types of boundary value problems of anti-plane elasticity: Neumann-type and Dirichlet-type. The domain under consideration contains an inclusion with material parameters different from those of the background. We assume that the domain and the inclusion are bounded by circles. In other words, we seek for the solutions $u$ and $v$ to
\begin{equation}\label{e0}
\begin{cases}
\ds\Delta u=0\quad&\mbox{in } D\cup \left(\Om\backslash{\overline{D}}\right),\\
\ds u|_+=u|_- \quad &\mbox{on }\partial D,\\[1mm]
\ds\frac{\partial u}{\partial \nu}\Big|_+=k\frac{\partial u}{\partial \nu}\Big|_- \quad &\mbox{on }\partial D,\\[2mm]
\ds\frac{\partial u}{\partial \nu}\Big|_-=g \quad &\mbox{on }\partial \Om,
\end{cases}\quad\mbox{and}\quad
\begin{cases}
\ds\Delta v=0\quad&\mbox{in } D\cup \left(\Om\backslash{\overline{D}}\right),\\
\ds v|_+=v|_- \quad &\mbox{on }\partial D,\\[1mm]
\ds\frac{\partial v}{\partial \nu}\Big|_+=k\frac{\partial v}{\partial \nu}\Big|_- \quad &\mbox{on }\partial D,\\[2mm]
\ds v|_- =g_d \quad &\mbox{on }\partial \Om,
\end{cases}
\end{equation}
where $\Om$ and $D$ are non-concentric disks such that $\overline{D}\subset \Om$, $k$ is a positive real number, and $\nu$ is the outward unit normal vector to each circle. Both of $g$ and $g_d$ are mean-zero functions on $\p \Om$ with some requirements in regularity, which will be discussed later. Temporarily assuming that $g$ and $g_d$ are continuous, Lax-Milgram theorem verifies that \eqnref{e0} are weakly solvable and the weak solutions in the Sobolev space $W^{1,2}(\Omega)$ are unique. The uniqueness of $u$ is up to additive constant.

Physically, $k$ denotes the ratio of the conductivity in $D$ to the conductivity in $\Om\setminus\overline{D}$. Also, we let $\ep$ denote the distance between the inclusion $D$ and $\p \Om$. The main purpose of this article is to analyze the blow-up behavior of $|\nabla u|$ and $|\nabla v|$ as $\ep$ tends to zero, where we fix the exterior domain $\Om$ but translate $D$ along the line passing through the centers of $\Om$ and $D$. The full characterization of the blow-up phenomenon in terms of $\ep$, $k$, $g$, $g_d$ and the radii of the two disks is provided in Theorem \ref{thm:imgCharge} and Theorem \ref{thm:imgCharge:u}. For the operators $\mathcal{S}_{\p \Om}$ and $\mathcal{D}_{\p \Om}$ defined in \eqnref{def:scal} and \eqnref{def:dcal}, it turns out that the harmonic functions
\beq\label{def:H}H=-2\mathcal{S}_{\p \Om}[g]\quad\mbox{and}\quad H_d=2\mathcal{D}_{\p\Om}[g_d]\quad\mbox{in }\Om,\eeq
which are the solutions to
\beq\label{eq:solH:woD}\begin{cases}
\ds\Delta H=0\quad&\mbox{in } \Om,\\[1mm]
\ds\frac{\partial H}{\partial \nu}\Big|_-=g \quad &\mbox{on }\partial \Om
\end{cases}\quad\mbox{and}\quad\begin{cases}
\ds\Delta H_d=0\quad&\mbox{in } \Om,\\[1mm]
\ds H_d=g_d \quad &\mbox{on }\partial \Om
\end{cases}\eeq
are used significantly in the characterization.

The methodology we take is analogous to that of \cite{imageCharge2018}, which provides a thorough analysis of blow-up feature of the gradient with the same type of interface conditions in the core-shell type geometry.
There, the domain is the free space and the far-field behavior of the solution is given to be an entire function.
The main observation in that paper is that only the two real numbers, which are extracted linearly from the incident far-field, contribute to the gradient under the approximation of $O(1)$, which is uniform in $\ep$ and $k$.
We show in this article that an analogous argument is valid for the solutions to \eqnref{e0} under suitable assumptions for the regularity of $g$ and $g_d$. Although the method of approximation is the same for both problems in \eqnref{e0}, we proved that the conditions for $\|\nabla u\|_{\infty}$ and $\|\nabla v\|_\infty$ to blow up are complementary to each other in some sense; see Table \ref{tab:CBUR} for the details.

\begin{table}[h]

\centering 
\begin{tabular}{l | c c c c}
\hline\\[0.5ex] 
$w$;\hfill $(C_1,C_2)$ & $\|\nabla w\cdot\mathbf{e}_\xi \|_{L^\infty(D)}$ & $\|\nabla w\cdot\mathbf{e}_\theta \|_{L^\infty(D)}$ &$\|\nabla w\cdot\mathbf{e}_\xi \|_{L^\infty(\Om\backslash\overline{D})}$ & $\|\nabla w\cdot\mathbf{e}_\theta \|_{L^\infty(\Om\backslash\overline{D})}$ 
\\ [0.5ex]\hline\hline\\ 
$=u$;  \hfill $\ne(0,0)$ & $\ds O\left(\frac{1}{k+\sqrt{\ep}}\right)$ & $\ds O\left(\frac{1}{k+\sqrt{\ep}}\right)$ & $O(1)$ & $\ds O\left(\frac{1}{k+\sqrt{\ep}}\right)$
\\[0.5ex]\hline\\

$=v$;\hfill $\ne(0,0)$ & $O(1)$ & $O(1)$ & $\ds O\left(\frac{1}{\frac{1}{k}+\sqrt{\ep}}\right)$ & $O(1)$

\\[0.5ex]\hline\\

$=u, v$;$\quad$ $=(0,0)$ & $O(1)$ & $O(1)$ & $O(1)$ & $O(1)$
\\[0.5ex]\hline\hline
\end{tabular}
\caption{\label{tab:CBUR}The above classification of blow-up rate is the main goal of this article. The parameters $(C_1,C_2)$ will be defined as a pair of real numbers extracted from $g$ or $g_d$; see Equations \eqnref{def:c1c2} and \eqnref{def:c3c4}. Also, $\{\mathbf{e}_\xi,\mathbf{e}_\theta\}$ is an orthonormal basis of $\mathbb{R}^2$ to be defined later; see Equations \eqnref{def:exet}.} 
\end{table}

%

The remainder is organized as follows. In Section 2, a coordinate system is introduced, which well describes the domain's geometry. Next, in Section 3, the solutions $u$ and $v$ are decomposed into two parts, only one of which contributing to the stress concentration. Finally, in Section 4, we derive the asymptotic formulas that explicitly show how each condition of the problem is related to the gradient blow-up. For each of Sections 3 and 4, the procedure is skipped for one of $u$ and $v$ due to the similarity.

\section{Geometry of the domain}\label{sec:shifted}
\subsection{Translation by $\ep$-dependent distance}
We may assume without loss of generality that 
\beq\label{def:domains}\Om=\left\{\Bx\in\RR^2:|\Bx-(r_e,0)|<r_e\right\}\quadand D=\left\{\Bx\in\RR^2:|\Bx-(r_i+\ep,0)|<r_i\right\}\eeq
for some $r_i,r_e>0$ such that $0<r_i<r_e$. Note that $\Om$ is located independently of $\ep$.

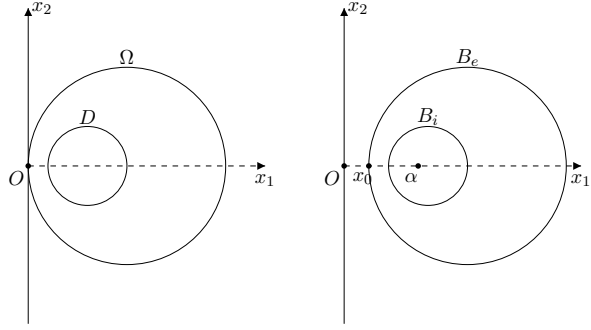
\begin{figure}[h!]
\begin{center}
\scalebox{0.7}{
\begin{tikzpicture}[scale=0.75]
\draw (3.125, 0) circle (2.5);
\draw (2.125, 0) circle (1);
\draw (0, -4.0) -- (0, 4.0);
\fill (0, 4) -- (-0.1, 3.8) -- (0.1, 3.8);

\draw[dashed] (0, 0) -- (6.0, 0);
\fill (6,0) -- (5.8, 0.1) -- (5.8, -0.1);

\draw (0.35, 4) node {$x_2$};
\draw (6, -0.35) node {$x_1$};

\draw (1.7, -0.3) node {$\alpha$};
\draw (0.475, -0.3) node {$x_0$};
\draw (3.125, 2.75) node {$B_e$};
\draw (2.125, 1.25) node {$B_i$};

\draw (-0.3, -0.3) node {$O$};
\fill (1.875, 0) circle (0.07);
\fill (0, 0) circle (0.07);
\fill (0.625, 0) circle (0.07);

\draw (-5.5, 0) circle (2.5);
\draw (-6.5, 0) circle (1);
\draw (-8, -4.0) -- (-8, 4.0);
\fill (-8, 4) -- (-8.1, 3.8) -- (-7.9, 3.8);

\draw[dashed] (-8, 0) -- (-2, 0);
\fill (-2,0) -- (-2.2, 0.1) -- (-2.2, -0.1);

\draw (-7.65, 4) node {$x_2$};
\draw (-2, -0.35) node {$x_1$};

\draw (-5.5, 2.75) node {$\Omega$};
\draw (-6.5, 1.25) node {$D$};

\draw (-8.3, -0.3) node {$O$};
\fill (-8, 0) circle (0.07);

\end{tikzpicture}
}
\end{center}
\caption{\label{fig:trans}Translation by the distance $x_0=O(\epsilon)$ from the left to the right}
\end{figure}


We then translate $\Om$ and $D$ as in Figure \ref{fig:trans} to obtain the auxiliary disks $B_e$ and $B_i$:
\beq\label{def:bebi}B_e = (x_0,0)+\Om\quadand B_i=(x_0,0)+D\eeq
 with
\begin{equation}\label{def:Ce:Ci}
c_i=\frac{r_e^2-r_i^2-(r_e-r_i-\ep)^2}{2(r_e-r_i-\ep)},\quad c_e=c_i+r_e-r_i-\ep\quad\mbox{and}\quad x_0=c_e-r_e.\end{equation}
In other words, $B_e$ and $B_i$ are the two open disks of radius $r_i$ and $r_e$ centered at $(c_i,0)$ and $(c_e,0)$. The motivation for the translation can be found in Subsection \ref{section:bipolar}.

\subsection{Bipolar coordinates}\label{section:bipolar}

\begin{figure}[h!]
\begin{center}
\scalebox{0.8}{
\begin{tikzpicture}[scale=0.75]
\draw (3.125, 0) circle (2.5);
\draw (2.125, 0) circle (1);
\draw (-3.125, 0) circle (2.5);
\draw (-2.125, 0) circle (1);
\draw (0, -5.0) -- (0, 5.0);

\draw[dashed] (0, 0) circle (1.875);
\draw[dashed] (0, 1.875) circle (2.65165042945);
\draw[dashed] (0, -1.875) circle (2.65165042945);
\draw[dashed] (-6.0, 0) -- (6.0, 0);

\draw (1.7, -0.3) node {$\mathbf{p}_2$};
\draw (-1.7, -0.3) node {$\mathbf{p}_1$};
\draw (-0.3, -0.3) node {$O$};
\draw (-6.5, 4.0) node {${z}$-plane:};
\draw (-5.3, 3.2) -- (-5.3, 4.7);
\draw (-5.3, 3.2) -- (-3.8, 3.2);
\draw (-4.0, 3.5) node {$\Re$};
\draw (-5.0, 4.5) node {$\Im$};
\fill (-5.3, 4.7) -- (-5.4, 4.5) -- (-5.2, 4.5);
\fill (-3.8, 3.2) -- (-4.0, 3.1) -- (-4.0, 3.3);
\fill (1.875, 0) circle (0.07);
\fill (-1.875, 0) circle (0.07);
\fill (0, 0) circle (0.07);

\draw (-3.6, 3.0) -- (-3.6, 5.0);
\draw (-7.5, 5.0) -- (-3.6, 5.0);
\draw (-7.5, 5.0) -- (-7.5, 3.0);
\draw (-3.6, 3.0) -- (-7.5, 3.0);

\end{tikzpicture}
}
\end{center}
\caption{\label{fig:coord}$\xi$-level curves (filled) and $\theta$-level curves (dashed) of the bipolar coordinate system}
\end{figure}
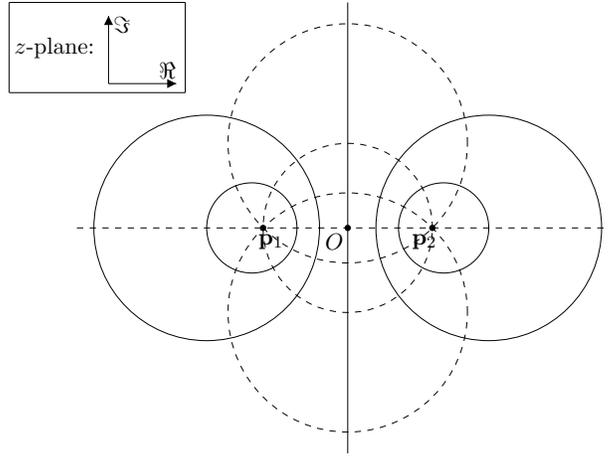
We will set the bipolar coordinate system $(\xi, \theta) \in \mathbb{R} \times (-\pi, \pi]$ by locating the two poles at $\Bp_1=(-\alpha,0)$ and $\Bp_2=(\alpha,0)$, where the constant $\alpha>0$ will be given later. Namely, any point $\Bx=(x_1,x_2)$ in the Cartesian coordinates corresponds to the bipolar coordinates $(\xi,\theta)$ by
	\beq\label{def:bipolar}
	e^{\xi+\rm{i}\theta}=\frac{\alpha+z}{\alpha-z}\quad\mbox{with } z=x_1+{\rm{i}}x_2.\eeq
	In other words, it holds that
	\begin{align}\label{bipolar:z}
	z&=\alpha \frac{e^{\xi+{\rm i}\theta}-1}{e^{\xi+{\rm i}\theta}+1}= \alpha \frac{\sinh \xi}{\cosh \xi + \cos \theta}+{\rm i} \alpha \frac{\sin \theta}{\cosh \xi + \cos \theta}.
	\end{align}
	We will write either $\Bx=\Bx(\xi,\theta)$ or $z=z(\xi,\theta)$ to represent the points in the complex plane, if necessary. The coordinate curves of the bipolar coordinate system are circles or parts of lines; see Figure \ref{fig:coord}.
%

Next, we make a detour to define the reflection across a circle of radius $r$ and center $\Bc$ by
$$\Bx\mapsto \Bc+\frac{r^2(\Bx-\Bc)}{|\Bx-\Bc|^2}.$$
In particular, we consider the reflection $R$ across the level curve $\{\xi=\xi_0\}$ for any $\xi_0\ne0$; one can deduce from \eqnref{bipolar:z} that the level curve is a circle. The proof of \eqnref{reflection:bipolar} is in \cite[Lemma 2.1]{imageCharge2018}.
\beq\label{reflection:bipolar}
R\left(\Bx(\xi, \theta)\right) = \Bx(2\xi_0 - \xi,\, \theta)\quad\mbox{for all }(\xi,\theta)\neq(2\xi_0,\pi).
\eeq

\smallskip

For example, we set $R_e$ and $R_i$ to be the reflections with respect to $\p B_e$ and $\p B_i$, respectively. Then the combined function $R_e\circ R_i$ has two fixed points $(\pm\alpha,0)$ with
\beq\label{def:alpha}
\alpha=\frac{\sqrt{\epsilon(2r_i+\epsilon)(2r_e-\epsilon)(2r_e-2r_i-\epsilon)}}{2(r_e-r_i-\epsilon)}.
\eeq
With $\alpha$ set as in \eqnref{def:alpha}, the two circles $\p B_e$ and $\p B_i$ are $\xi$-level curves of \eqnref{def:bipolar} with $\xi$-values 
\beq\label{def:xixe}\xi_e=\frac{1}{2}\ln\left(\frac{c_e+\alpha}{c_e-\alpha}\right)\quadand \xi_i=\frac{1}{2}\ln\left(\frac{c_i+\alpha}{c_i-\alpha}\right),\eeq
respectively.

Finally, note that $\xi$ and $\theta$ in \eqnref{def:bipolar} defines an orthogonal coordinate system in $\mathbb{R}^2\backslash\{\mathbf{p}_1,\mathbf{p}_2\}$. We can then consider the scale factor $h(\xi,\theta)$ defined by
\beq\label{eqn:h}
\left|\frac{\partial z}{\partial \xi}\right|=\left|\frac{\partial z}{\partial \theta}\right|=\frac{1}{h(\xi,\theta)} 
\quadwith
h(\xi,\theta)=\frac{\cosh\xi+\cos\theta}{\alpha}.
\eeq
For any constant $\xi_0\ne0$, we let $\nu$ be the outward unit normal vector to the circle $\{\xi=\xi_0\}$ 
and rotate $\nu$ by $\frac{\pi}{2}$-radian to get $T$. Then the gradient along the bipolar coordinate curves are simply
\begin{align}
\frac{\p }{\p \nu}\bigg|_{\Bx(\xi_0,\theta)}\label{N:bipolar}
&=-\operatorname{sgn}(\xi_0)h(\xi_0,\theta)\frac{\p }{\partial\xi}\bigg|_{\Bx(\xi_0,\theta)}\quadand\\
\frac{\p }{\p T}\bigg|_{\Bx(\xi_0,\theta)} \label{T:bipolar}
&= -\mbox{sgn} (\xi_0) h(\xi_0,\theta) \frac{\p }{\p \theta}\bigg|_{\Bx(\xi_0,\theta)}.
\end{align}
	

\section{Formulation by layer potentials}\label{sec:formulation}
\subsection{Representation of $u$ in terms of single-layer potentials}

For a Lipschitz domain $U$, we define the single layer potential for $\phi\in L^2(\p U)$ as
\beq\label{def:scal}\Scal_{\p U}[\phi](\Bx)=\frac{1}{2\pi}\int_{\p U}\ln|\mathbf{x}-\zeta|\phi(\zeta)d\sigma(\zeta).\eeq
In addition, we denote the collection of square-integrable mean-zero functions on $\p U$ by $L^2_0(\p U)$. In particular, let $B$ be a disk. Then, we refer the readers to \cite{book} that for each $\phi\in L_0^2(\p B;\mathbb{R})$, $$\ds\lim_{t\to0}\mathcal{S}_{\p B}[\phi](\Bx+t\nu(\Bx))\quad\mbox{converges for a.e. }\Bx\in\p B$$ and
\beq\label{eq:jump:nd}\nu(\Bx)\cdot\lim_{t\to0+}\nabla\mathcal{S}_{\p B}[\phi](\Bx\pm t\nu(\Bx))=\pm\frac{1}{2}\phi(\Bx)\quad \mbox{for a.e. }\Bx\in \p B.\eeq
We deduce from \eqnref{eq:jump:nd} that $H$ defined in \eqnref{def:H} solves the boundary value problem \eqnref{eq:solH:woD}.

As is well-known {\cite[Theorem 3.1]{KS1996LPTICP}}, the solution $u$ to \eqnref{e0} up to constant satisfies 
\beq\label{u:Scal}
u(\Bx)=\Scal_{\p D}[\varphi_i](\Bx)+\Scal_{\p \Om}[\varphi_e](\Bx),\quad\Bx\in B_e\eeq
for some density functions $(\varphi_i,\varphi_e)\in L^2_0(\p D)\times L^2_0(\p \Om)$. By the boundary conditions in \eqnref{e0},
\begin{equation}\label{eqn:density:transmission}\begin{bmatrix}
\ds\frac{1}{2\tau}I & \ds-\frac{\p}{\p\nu} \Scal_{\partial \Om}\\[2.5mm]
\ds\frac{\partial}{\partial \nu}\Scal_{\p D}&\ds-\frac{1}{2}I
\end{bmatrix}\begin{bmatrix}
\ds \varphi_i \\ \ds \varphi_e
\end{bmatrix}=\begin{bmatrix}
\ds 0 \\\ds  {g}
\end{bmatrix}\quad\mbox{with }\tau=\frac{k-1}{k+1}.\end{equation}
In the following subsection, we will find an expression of $u$ in series by solving \eqnref{eqn:density:transmission}.


%

\subsection{Representation of $u$ in terms of repeated reflections}
Let $B$ be a disk centered at ${c}$ and $\nu$ be the outward unit normal vector on $\partial B$. If $v$ is harmonic in $B$ and continuous on $\overline{B}$, then we have
\beq\label{Scal:disk1}
\Scal_{\partial B}\left[\frac{\partial v}{\partial \nu}\Big|_{\partial B}^-\right](x)=\begin{cases}
\ds-\frac{1}{2}v(x)+\frac{v(c)}{2} & \quad\mbox{for }x\in B,\\[2mm]
\ds-\frac{1}{2}R_{\partial B}[v](x)+\frac{v(c)}{2} & \quad\mbox{for }x\in \mathbb{R}^2\backslash \overline{B}.
\end{cases}
\eeq
If $v$ is harmonic in $\mathbb{R}^2\backslash\overline{B}$, continuous on $\RR^2\setminus B$, and $\lim_{|\Bx|\to\infty}v(\Bx)=0$, then we have
\beq\label{Scal:disk2}
\Scal_{\partial B}\left[\frac{\partial v}{\partial \nu}\Big|_{\partial B}^+\right](x)=\begin{cases}
\ds\frac{1}{2}R_{\partial B}[v](x)&\quad\mbox{for }x\in B\\[2mm]
\ds\frac{1}{2}v(x)&\quad\mbox{for }x\in \mathbb{R}^2\backslash \overline{B}.
\end{cases}
\eeq
One can find the proof of \eqnref{Scal:disk1} and \eqnref{Scal:disk2} in \cite{AKL}, which basically owes to the uniqueness of the solutions to Neumann-type boundary value problems. Using \eqnref{Scal:disk1} and \eqnref{Scal:disk2}, we will express the density functions $\varphi_i$ and $\varphi_e$ in terms of repeated reflections.

To simplify the notation, we denote the series of reflections for a function, say $f$, as
\begin{align*}
R_{i}[f](\Bx)&:=f(R_i(\Bx)),\\
R_{e}R_{i}[f](\Bx)&: =f(R_{i}\circ R_{e}(\Bx)),\quad \Bx\in \RR^2.
\end{align*}
We define the other combinations in the same manner. In addition, we set
	\begin{equation}\label{def:xik}
	\begin{cases}
	\ds\xi_{i,n}=2n(\xi_i-\xi_e)+\xi_i,\\
	\ds\xi_{e,n} = 2n(\xi_i-\xi_e)+\xi_e
	\end{cases}
	\end{equation} 
	for each $n\in\NN\cup\{0\}$. 
	One can show $$\xi_i>\xi_e>0$$ directly from \eqnref{def:xixe}; thus, $\xi_{i,k}\geq\xi_i$ and $\xi_{e,k+1}\geq \xi_i$ for each $k\ge0$.%

\begin{lemma}\label{lemma:density:series} The solution $u$ to \eqnref{e0} admits the series expansion
\[u=\begin{cases}\ds H+\sum_{n=0}^\infty(-\tau)^{n+1}\Big[R_{\p D}(R_{\p\Om}R_{\p D})^n H+(R_{\p\Om}R_{\p D})^{n+1} H\Big]&\mbox{in }\Om\setminus\overline{D},\\[2mm]
\ds H+\sum_{n=0}^\infty(-\tau)^{n+1}\Big[(R_{\p\Om}R_{\p D})^n H+(R_{\p\Om}R_{\p D})^{n+1} H\Big]&\mbox{in }D.\end{cases}\]
\end{lemma}

%
%
\begin{proof}
For simplicity, we use only in this proof the notation $$\widetilde{f}(\Bx)=f(\Bx-(x_0,0))\quad\mbox{for any function }f$$ for the translation of domains as in \eqnref{def:bebi}. In addition, we set
$$\frac{\partial \Scal_{\partial B_e}}{\partial \nu_i}[f]=\frac{\partial }{\partial \nu_i}\left(\Scal_{\partial B_e}[f]\right)\quad\mbox{for }\nu_i=\nu_{B_i}\mbox{ and }f\in L^2(\p B_e),$$
and likewise for $\frac{\partial \Scal_{\partial B_i}}{\partial \nu_e}$. 
From \eqnref{Scal:disk1} and \eqnref{Scal:disk2}, we first observe that $(\varphi_i,\varphi_e)$ obtained from
\begin{align}
\label{eq:density:ref:i}\widetilde{\varphi}_i&=-4\tau\sum_{n=0}^\infty\left(-\tau\right)^{n}\frac{\partial}{\partial\nu_i}\big((R_eR_i)^n\Scal_{\partial B_e}[\widetilde{g}]\big)\quad\mbox{and}\\
\label{eq:density:ref:e}\widetilde{\varphi}_e&=-2\widetilde{g}+4\tau\sum_{n=0}^\infty\left(-\tau\right)^{n}\frac{\partial }{\partial \nu_e}\big(R_i(R_eR_i)^n\Scal_{\partial B_e}[\widetilde{g}]\big)
\end{align}
formally satisfy Eq.\;\eqnref{eqn:density:transmission}. In the bipolar coordinates,
\begin{align}
\label{eq:density:i}\ds \widetilde{\varphi}_i(\mathbf{x}(\xi_i,\theta))
&=2\tau\sum_{n=0}^\infty(-\tau)^n\frac{h(\xi_i,\theta)}{h(\xi_{i,n},\theta)}\frac{\partial \widetilde{H}}{\partial\nu}(\Bx(\xi_{i,n},\theta)),\\
\label{eq:density:e}\ds \widetilde{\varphi}_e(\mathbf{x}(\xi_e,\theta))
&=-2\widetilde{g}-2\tau\sum_{n=0}^\infty(-\tau)^n\frac{h(\xi_e,\theta)}{h(\xi_{e,n+1},\theta)}\frac{\partial \widetilde{H}}{\partial\nu}(\Bx(\xi_{e,n+1},\theta)),
\end{align} as can be seen from \eqnref{def:H}, \eqnref{reflection:bipolar}, \eqnref{N:bipolar} and \eqnref{def:xik}.

Similarly to \cite[Lemma 3.3]{imageCharge2018}, we can prove that \eqnref{eq:density:i} and \eqnref{eq:density:e} converge uniformly in $\p B_i$ and $\p B_e$ respectively, using the facts that $\left|\tau\right|<1$, 
\beq \label{ineq:hfrac}
0<\frac{h(\xi,\theta)}{h(\tilde{\xi},\theta)}\leq 1\quad\mbox{for any }\xi\leq\tilde{\xi},
\eeq
and that $\|\nabla \widetilde{H}\|_{L^\infty(B_e)}$ is bounded independently of $\ep$. Therefore, we can use \eqnref{Scal:disk1} and \eqnref{Scal:disk2} term-by-term in the expressions \eqnref{eq:density:ref:i} and \eqnref{eq:density:ref:e} to complete the proof.
\end{proof}

\begin{figure}[h!]
\centering
\includegraphics[width=\textwidth]{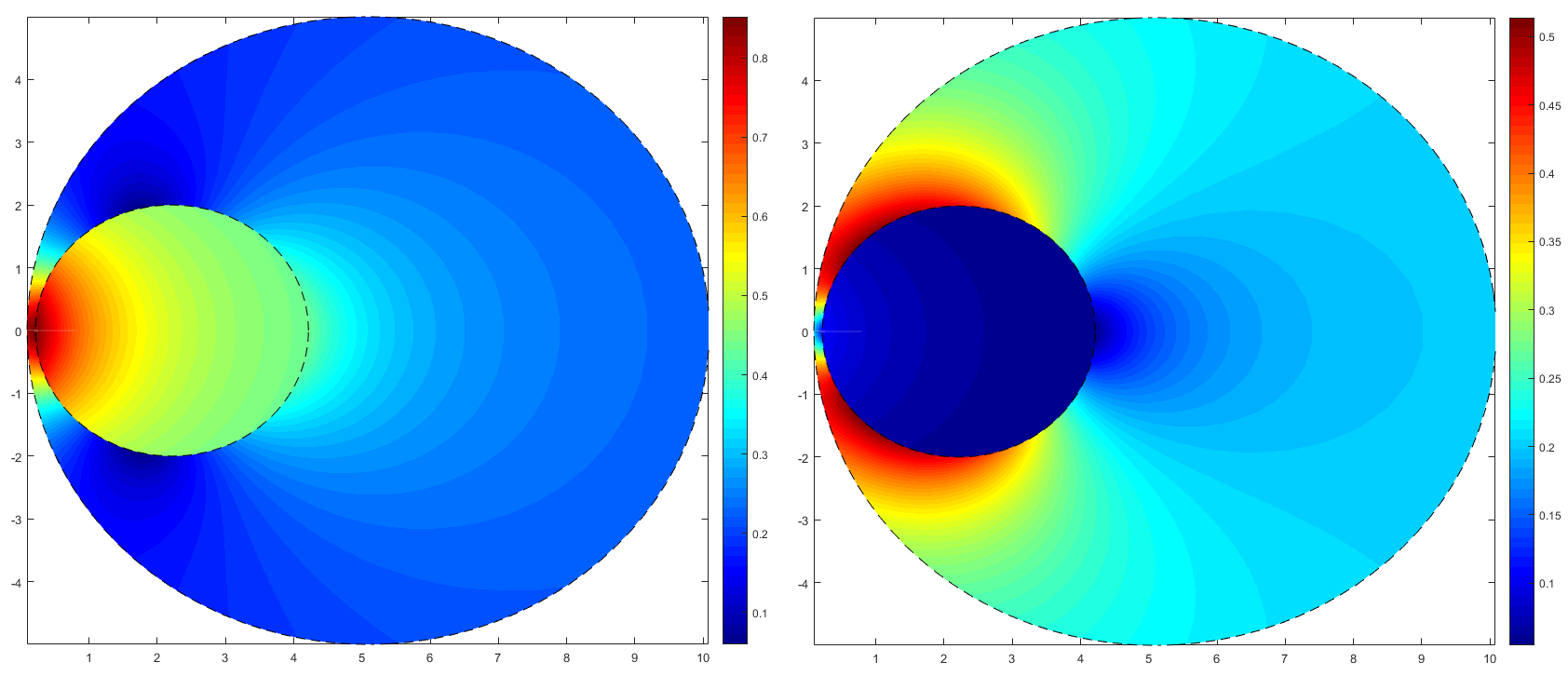}
\caption{Using Lemma \ref{lemma:density:series} and \ref{lem:vrep:ref}, we plot $|(\nabla u)(\cdot-(x_0,0))|$ (left) and $|(\nabla v)(\cdot-(x_0,0))|$ (right) with respect to the boundary conditions $g(t)=\sin t$ and $g_d(t)=\sin t$, respectively, where $t$ parametrizes $\p \Om$ by $r_e+r_e e^{it}$. Here, the geometric parameters are set to be $r_i=2$, $r_e=5$ and $\ep=1/8$. The relative conductivity of the inclusion $D$ is $k=1/8$ for $u$ and $k=8$ for $v$.}
\end{figure}

\begin{remark}\label{remark:series:regularity}
\begin{itemize}
\item In the remainder of this article, we assume that \beq\label{cond:H} H\in C^2(\overline{\Om}).\eeq
For instance, it is sufficient to assume
\beq\label{cond:g}
g\in C^{1,\delta}(\p \Om)\quad\mbox{for some }\delta>1/2,
\eeq 
as $H$ is a harmonic function in the disk $\Om$ satisfying 
$$\pd{H}{\nu}=g\quad\mbox{on }\p\Om.$$ One can show that the condition \eqnref{cond:g} implies \eqnref{cond:H} using Bernstein's theorem, which states that if a function $f$ defined on a circle is H\"{o}lder continuous of order $>1/2$, then the Fourier series of $f$ converges absolutely.
 
\item Since $|\tau|<1$ and $H\in C^2(\overline{\Om})$, differentiating \eqnref{eq:density:i} and \eqnref{eq:density:e} term-by-term gives uniformly convergent series, which implies $\varphi_i\in C^1(\p B_i)$ and $\varphi_e\in C^1(\p B_e)$. Moreover,
	\beq\label{eq:Fourier:absConv}\sup_{\p B_i}\left|\frac{d\varphi_i}{dT}\right|,\;\sup_{\p B_e}\left|\frac{d\varphi_e}{dT}\right|\le \left(2+4\sum_{m=1}^\infty m |\tau|^m\right)\|H\|_{C^2(\Om)}<\infty,\eeq
	where $\|H\|_{C^2(\Om)}$ is the sum of the supremums of $|H|$, $|DH|$ and $|D^2H|$ in $\Om$.
\end{itemize}

\end{remark}

\begin{cor}\label{boundedness1}
		Suppose that $g\in C^{1,\delta}(\p\Om)$. If $k$ is bounded but not small, then $|\nabla u|$ is uniformly bounded in $\Om$ independently of $(\ep,k)$.
	\end{cor}
	\begin{proof}
	The assumption for $k$ is simplified into
	\[|\tau|\le \tau_0<1\]
	for some uniform constant $\tau_0$. On the other hand, if a $2\pi$-periodic function $\varphi$ is H{\"o}lder continuous of order $\delta>1/2$, namely
	$$|\varphi(t_1)-\varphi(t_2)|\le C_\varphi|t_1-t_2|^\delta,$$
	then Bernstein's theorem states that
	\beq\label{ineq:Fourier:ufBd}\sum_{n=-\infty}^\infty|\hat{\varphi}(n)|\le C_\delta C_\varphi<\infty,\quad\mbox{where}\quad\hat{\varphi}(n)=\frac{1}{2\pi}\int_{-\pi}^\pi\varphi(t)e^{-int}dt\eeq
for some constant $C_\delta$ depending only on $\delta$. Combining \eqnref{eq:Fourier:absConv} with \eqnref{ineq:Fourier:ufBd}, we conclude that both $$\varphi(t)=\varphi_i(c_i+r_ie^{it})\quad\mbox{and}\quad \varphi(t)=\varphi_e(c_e+r_ee^{it})$$ satisfy
	\beq\label{ineq:Fourier:particularUB}\sum_{n=-\infty}^\infty|\hat{\varphi}(n)|\le C_{\delta=1}\left(2+4\sum_{m=1}^\infty m \tau_0^m\right)\|H\|_{C^2(\Om)}.\eeq
	Since $\|\nabla u\|_{L^\infty(\Om)}$ is bounded by the sum of the absolute Fourier series of $\varphi_i$ and $\varphi_e$, we complete the proof as the right hand side of \eqnref{ineq:Fourier:particularUB} is independent of $(\ep,k)$.
	\end{proof}

\subsection{Blow-up contribution of $H=-2\mathcal{S}_{\p\Omega}[g]$ to $\nabla u$}

In this subsection, we extend Corollary \ref{boundedness1} to include the limiting behavior as $k\to 0$ and $k\to\infty$. For the remainder of this article, we set the two parameters \beq\label{def:c1c2} C_1:=-\pd{H}{\nu}(0,0)=\pd{H}{x_1}(0,0)\quad \mbox{and}\quad C_2:=-\pd{H}{T}(0,0)=\pd{H}{x_2}(0,0),\eeq
where $H$ is defined by \eqnref{def:H}. In addition, we adopt the notation
\[r_*:=\sqrt{\frac{2r_ir_e}{r_e-r_i}}\quad\mbox{so that}\quad \alpha=r_*\sqrt{\ep}+O(\ep\sqrt{\ep}).\]

\begin{lemma}{\rm(\cite[Lemma 3.3]{imageCharge2018})} \label{coshsum} 
Let $N_\ep = \frac{r_*}{\sqrt{\epsilon}}$, then we have
		\begin{equation}\label{ineq:lem33}
		\begin{cases}
		\ds\sum_{m=0}^{N_\ep} \frac{h(\xi_i,\theta)}{h(\xi_{i,m},\theta)}
		\le  \frac{C}{|\Bx(\xi_i,\theta)|} \qquad\mbox{for all }|\theta|>\frac{\pi}{2},\\[4mm]
		\ds\sum_{m=0}^{N_\ep} \frac{h(\xi_e,\theta)}{h(\xi_{e,m+1},\theta)}
		\le  \frac{C}{|\Bx(\xi_e,\theta)|} \qquad\mbox{for all }|\theta|>\frac{\pi}{2},
		\end{cases}
		\end{equation}
where $C$ is a constant number independent of $(\ep,\theta)$.  

			\end{lemma}

\begin{lemma}\label{lemma:boundedness} Suppose that $g\in C^{1,\delta}(\p\Om)$. If $C_1=C_2=0$, then we have
		\beq\label{eq:lindep:c1c20}\lVert \nabla u\rVert_{L^\infty (\Om)} \le C\eeq
		for some constant $C$ independent of $(\ep,k)$.
	\end{lemma}
\begin{proof}
Recall from Remark \ref{remark:series:regularity} that $H\in C^2(\overline{\Om})$. Hence, the assumption $C_1=C_2=0$ implies
\beq\label{eq:Htild:bigO}
 \nabla H(\Bx-(x_0,0))=O(|\Bx|+\ep),\quad \Bx\in B_e
\eeq
where $x_0$ defined in \eqnref{def:Ce:Ci} satisfies $x_0=O(\ep)$. We can prove \eqnref{eq:lindep:c1c20} analogously to \cite[Lemma 3.6]{imageCharge2018}. Lemma \ref{coshsum} is essential in the derivation; see Appendix \ref{proof:boundedness} for the rest of the proof.
\end{proof}

\subsection{Blow-up contribution of $H_d=2\mathcal{D}_{\p\Omega}[g_d]$ to $\nabla v$}

In this subsection, we consider the second problem in \eqnref{e0} with Dirichlet-type boundary condition $v=g_d$ on $\p\Om$. It turns out that all the above arguments for $u$ can be applied for $v$ as well. In the application for $v$, we use the representation
\beq\label{rep:v} v=\mathcal{D}_{\p \Om}[\varphi_e^d]+\Scal_{\p D}[\varphi_i^d]\quad\mbox{in }\Om,\quad(\varphi_e^d,\varphi_i^d)\in L_0^2(\p\Om)\times L_0^2(\p\Om),\eeq
where the operator $\mathcal{D}_{\p U}$ is called to be the double-layer potential and is defined by
\beq\label{def:dcal} \mathcal{D}_{\p U}[\phi]=\frac{1}{2\pi}\int_{\p U}\frac{\p}{\p\nu_\zeta}\left(\ln|\mathbf{x}-\zeta|\right)\phi(\zeta)d\sigma(\zeta),\quad \phi\in L^2(\p U)\eeq
for bounded domains $U\subset\mathbb{R}^2$ with Lipschitz boundary. In this case, the harmonic function $H_d$ defined in \eqnref{def:H} replaces the role of $H$. Since the derivation using reflections and the proof of convergence are quite similar, we will directly present the results without detailed proofs.
\begin{lemma}\label{lem:vrep:ref} If $0<k\ne1$, the solution $v$ to \eqnref{e0} admits the representation in series as
\[v=\begin{cases}\ds H_d+\sum_{n=0}^\infty\tau^{n+1}\left[(R_{\p\Om}R_{\p D})^{n+1}H_d-R_{\p D}(R_{\p\Om}R_{\p D})^nH_d\right]&\mbox{in }\Om\setminus\overline{D},\\[2mm]
\ds H_d+\sum_{n=0}^\infty\tau^{n+1}\left[(R_{\p\Om}R_{\p D})^{n+1}H_d-(R_{\p\Om}R_{\p D})^nH_d\right]&\mbox{in }D.\end{cases}\]
\end{lemma}
\begin{cor} Suppose that $g_d\in C^{2,\delta}(\Om)$ for some $\delta>1/2$. If $k$ is bounded but not small, then $|\nabla v|$ is uniformly bounded in $\Omega$ independently of $(\ep,k)$.
\end{cor}

As for the Dirichlet-type boundary condition, we set \beq\label{def:c3c4} C_1:=-\pd{H_d}{\nu}(0,0)=\pd{H_d}{x_1}(0,0)\quad \mbox{and}\quad C_2:=-\pd{H_d}{T}(0,0)=\pd{H_d}{x_2}(0,0),\eeq
where $H_d$ is defined by \eqnref{def:H}. Lemma \ref{coshsum} proves the following lemma, which states that only the linear term of $H_d$ contributes to the blow-up of $\nabla v$.

\begin{lemma}\label{lemma:boundedness:d} Suppose that $g_d\in C^{2,\delta}(\p\Om)$. If $C_1=C_2=0$, then we have
		\beq\label{eq:lindep:c3c4}\lVert \nabla v\rVert_{L^\infty (\Om)} \le C\eeq
		for some constant $C$ independent of $(\ep,k)$.
	\end{lemma}

\section{Image charge formula for the gradient blow-up term}
In this section, we directly find the asymptotic solution in the Fourier basis corresponding to the bipolar coordinate system using the results in Section 3. Throughout the rest of the article, $\Om$ and $D$ are relocated to avoid lengthy notation, as
\beq\label{def:trans:domains}\Omega=\{\mathbf{x}(\xi,\theta):\ \xi=\xi_e\}\quad\mbox{and}\quad D=\{\mathbf{x}(\xi,\theta):\ \xi=\xi_i\}.\eeq
Since $\nabla u$ turns out to blow up more frequently than $\nabla v$, we start with estimating $ v$ for simplicity.

We conclude from Lemma \ref{lemma:boundedness} and Lemma \ref{lemma:boundedness:d} that only the linear terms in $H$ and $H_d$ contribute to the gradient blow-up. Hence, as we approximate $\nabla v$ in $O(1)$ uniformly in $(\ep,k)$, we may assume that 
\beq\label{Hd:linear}H_d(x_1,x_2)=C_0+C_1x_1+C_2x_2 \quad\mbox{in }\overline{\Om}
\eeq
for some constant $C_0$, without loss of generality. Note that the constant numbers $C_1$ and $C_2$ defined by \eqnref{def:c3c4} does not depend on $\ep$.

	\subsection{Integral representation of $v$}
The linear functions $x_1$ and $x_2$ admits the following expansion in bipolar coordinates:
\begin{align*}
x_1&=\alpha\left[1+2\sum_{n=1}^\infty (-1)^ne^{-n\xi}\cos n\theta\right],\\
x_2&=-2\alpha\sum_{n=1}^\infty (-1)^n e^{-n\xi}\sin n\theta.
\end{align*}
In addition, the general solution for the problem $\Delta f=0$ is derived from separation of variables:
\[f(\xi,\theta)=a_0+b_0\xi+\sum_{n=1}^\infty\left[(a_ne^{n\xi}+b_ne^{-n\xi})\cos n\theta+(c_ne^{n\xi}+d_ne^{-n\xi})\sin n\theta\right].\]
One can readily obtain the following lemma using the relation \eqnref{N:bipolar} of normal derivatives.
\begin{lemma}\label{lemma:tu:series}
 The solution ${v}$ to \eqnref{e0} with $v=H_d$ on $\p \Om$ for $H_d$ defined in \eqnref{Hd:linear} satisfies
\beq\label{eqn:tu:series}
v(\Bx)=C_1\left(x_1+\Re\{F(\mathbf{x})\}\right)
+C_2\left(x_2+\Im\{F(\mathbf{x})\}\right),
\eeq
where $F$ is defined in $\Om$ by
\begin{equation}\label{e5}
F\left(\mathbf{x}(\xi,\theta)\right)={\rm{const.}}+\begin{cases}
\ds\sum_{n=1}^\infty\left(A_ne^{n(\xi-i\theta-2\xi_i)}+B_ne^{n(-\xi-i\theta)}\right)&\mbox{for }\xi_e<\xi<\xi_i,\\[2mm]
\ds\sum_{n=1}^\infty\left(A_ne^{n(-\xi-i\theta)}+B_ne^{n(-\xi-i\theta)}\right)&\mbox{for }\xi>\xi_i
\end{cases}\end{equation}
with
\beq\label{def:AnBn} A_n=\frac{-2\alpha(-1)^n\tau}{1-\tau e^{-2n(\xi_i-\xi_e)}}\quadand B_n=\frac{2\alpha(-1)^n\tau e^{-2n(\xi_i-\xi_e)}}{1-\tau e^{-2n(\xi_i-\xi_e)}}.\eeq
\end{lemma}

%

\smallskip

Next, we approximate the series \eqnref{e5} using integrals. For convenience, we set 
\beq
\beta=\frac{r_*(-\ln |\tau|)}{4\sqrt{\ep}}\quad\mbox{whenever }k>1.
\eeq
\begin{definition}
Whenever $k>1$, we define a singular function $q_d:\Om\setminus\p D\to\mathbb{C}$ by
\beq
q_d(\Bx(\xi,\theta);\beta)=\begin{cases}
\ds L\left(e^{-(2\xi_i-2\xi_e+\xi)-i\theta};\beta\right)-L\left(e^{-(2\xi_i-\xi)-i\theta};\beta\right)\quad&\mbox{for }\xi_e<\xi<\xi_i,\\[2mm]
\ds L\left(e^{-(2\xi_i-2\xi_e+\xi)-i\theta};\beta\right)-L\left(e^{-\xi-i\theta};\beta\right)\quad&\mbox{for }\xi>\xi_i,
\end{cases}
\eeq
where $L(z;\beta)$ denotes the Lerch transcendent function
\beq
L(z;\beta)=-\int_0^\infty\frac{ze^{-(\beta+1)t}}{1+ze^{-t}}\,dt\quad\mbox{for }z\in\mathbb{C},\ |z|<1,\ \beta>0.
\eeq
We also set
\beq
P(z;\beta)=-z\frac{\partial}{\partial z}L(z;\beta)=\int_0^\infty\frac{ze^{-(\beta+1)t}}{(1+ze^{-t})^2}\,dt.
\eeq

\end{definition}

In the following, we approximate the series \eqnref{e5} in terms of $q_d$. We refer the readers to the appendix for the detailed proof of Lemma \ref{thm:singularLT}. To provide the crucial idea in proof of Lemma \ref{thm:singularLT}, we state Lemma \ref{lem:approxP} right before Lemma \ref{thm:singularLT}.

\begin{lemma}{\rm(\cite[Lemma 2.4]{LY:2015:ASCE})}\label{lem:approxP}
Fix $\xi_0>0$ and $0<|\tau|<1$, (i.e., $0<k\ne1$.) Then for each $\xi<\xi_0$ and $-\pi<\theta\le\pi$,
\beq\label{ineq:simple:kg1}\left|\xi_0\sum_{m=1}^\infty\left(|\tau|^{m-1}\frac{e^{-m\xi_0+\xi+i\theta}}{(1+e^{-m\xi_0+\xi+i\theta})^2}\right)-P\left(e^{-(\xi_0-\xi)+i\theta};\frac{-\ln|\tau|}{\xi_0}\right)\right|\le\frac{8\xi_0}{\cosh(\xi_0-\xi)+\cos\theta}\eeq
and
\beq\label{ineq:simple:kl1}\left|\xi_0\sum_{m=1}^\infty\left((-|\tau|)^{m-1}\frac{e^{-m\xi_0+\xi+i\theta}}{(1+e^{-m\xi_0+\xi+i\theta})^2}\right)\right|\le\frac{8\xi_0}{\cosh(\xi_0-\xi)+\cos\theta}.\eeq
\end{lemma}

%

%

\begin{lemma}\label{thm:singularLT}
Suppose $g_d\in C^{2,\delta}(\p\Om)$ for some $\delta>1/2$. The solution $v$ to \eqnref{e0} satisfies the following.
\begin{itemize}
\item[\rm(a)] For $0<k<1$, $\|\nabla v\|_{L^\infty(\Om)}$ is bounded independently of $\ep$ and $k$. 

\item[\rm(b)] For $1<k<\infty$, $v$ satisfies
\beq
v(\Bx)= \frac{r_*^2\tau}{2} \big[C_1\Re\{ q_d(\Bx;\beta)\}+ C_2\Im\{q_d(\Bx;\beta)\}\big]+r(\Bx),
\eeq
where $\|\nabla r\|_{L^\infty(\Om)}$ is bounded independently of $\ep$ and $k$.
\end{itemize}
\end{lemma}

	\subsection{Rate of gradient blow-up}

In order to simplify the representation in Lemma \ref{thm:singularLT} further, we introduce some properties of the function $P$, which directly follow from integration by parts:
\begin{equation}\label{e9}|P(e^{-s+i\theta};\beta)|\le\frac{1}{2\beta(\cosh s+\cos\theta)}\quad\forall s>0\end{equation}
and
\begin{equation}\label{e10}|P(e^{-s_2+i\theta};\beta)-P(e^{-s_1+i\theta};\beta)|\le\frac{s_2-s_1}{2(\cosh s_1+\cos\theta)}\quad\forall s_2>s_1>0.\end{equation}
Eq. \eqnref{e9} is used to derive the upper bound of stress concentration, and Eq. \eqnref{e10} simplifies the singular term of $\nabla v$.
As noticed in Section 1, we define the orthonormal basis $\{\mathbf{e}_\xi ,\mathbf{e}_\theta\}$ by
\beq\label{def:exet}
\mathbf{e}_\xi = \nabla\xi/|\nabla \xi|\quadand \mathbf{e}_\theta = \nabla\theta/|\nabla \theta|.
\eeq

\begin{theorem}\label{theo:grad} Suppose $g_d\in C^{2,\delta}(\p\Om)$ for some $\delta>1/2$. Then, $\|\nabla v\|_{L^\infty(D)}$ is bounded uniformly in $(\ep,k)$, whereas $\|\nabla v\|_{L^\infty(\Om\backslash\overline{D})}$ blows up only in the limits $\ep\to 0$ and $\tau\to1$, i.e., $k\gg1$. The gradient blow-up in $\Om\setminus\overline{D}$ is approximated by the asymptotic formulas as
\beq\begin{aligned}&\ds\nabla v(\mathbf{x}(\xi,\theta))\\[2mm]
&\ds=\frac{r_*\tau}{\sqrt{\ep}}(\cosh\xi+\cos\theta)\left[C_1\Re\left\{P\left(e^{-(2\xi_i-\xi)-i\theta};\beta\right)\right\}+C_2\Im\left\{P\left(e^{-(2\xi_i-\xi)-i\theta};\beta\right)\right\}\right]\mathbf{e}_\xi+O(1)\\[2mm]
&\ds=\frac{r_*\tau}{\sqrt{\ep}}(\cosh\xi+\cos\theta)\left[C_1\Re\left\{P\left(e^{-(\xi+2\xi_i)-i\theta};\beta\right)\right\}+C_2\Im\left\{P\left(e^{-(\xi+2\xi_i)-i\theta};\beta\right)\right\}\right]\mathbf{e}_\xi+O(1).\end{aligned}\eeq
\end{theorem}
\begin{proof} We may suppose $k>1$ using (a) of Lemma \ref{thm:singularLT}. Applying \eqnref{e10} to (\ref{e7}) and (\ref{e8}), we have
\[\nabla q_d\cdot\mathbf{e}_\xi=h(\xi,\theta)\begin{cases}
\ds 2P\left(e^{-(2\xi_i-\xi)-i\theta};\beta\right)+O\left(\frac{1}{h(\xi,\theta)}\right)&\mbox{or}\\[2mm]
\ds 2P\left(e^{-(\xi+2\xi_i)-i\theta};\beta\right)+O\left(\frac{1}{h(\xi,\theta)}\right)&\mbox{for }\xi_e<\xi<\xi_i,\\[2mm]
\ds O\left(\frac{1}{h(\xi,\theta)}\right)&\mbox{for }\xi>\xi_i,
\end{cases}\]
and
\[\nabla q_d\cdot\mathbf{e}_\theta=O(1)\quad\mbox{for }\xi>\xi_e.\]
Moreover, Eq. (\ref{e9}) shows
\[\left|2h(\xi,\theta)P\left(e^{-(2\xi_i-\xi)-i\theta};\beta\right)\right|\le\frac{1}{\alpha\beta}\frac{h(\xi,\theta)}{h(2\xi_i-\xi,\theta)}\le\frac{C}{|\ln \tau|}\]
for $\xi_e<\xi<\xi_i$, which implies that the gradient blow-up does not occur unless $k\to\infty$.
\end{proof}

From asymptotic analysis on the integral representation in Theorem \ref{theo:grad}, the optimal estimate for the point-wise divergence of $\nabla v$ follows. The proof is given in the appendix in detail.

\begin{cor}\label{prop:optEst} Suppose $g_d\in C^{2,\delta}(\p\Om)$ for some $\delta>1/2$. For some constant number $C>0$ independent of $g_d$, $\ep$ and $k$,
\[\|\nabla v\|_{L^\infty(\Om\backslash\overline{D})}\le C\frac{|C_1|+|C_2|}{\frac{1}{k}+\sqrt{\ep}}.\]
The upper bound is optimal in that whenever $C_1C_2=0$ and $\ep^{-1}=O(k^2)$,
\[\|\nabla v\|_{L^\infty(\Om\backslash\overline{D})}\ge \frac{|C_1|+|C_2|}{{C}\left(\frac{1}{k}+\sqrt{\ep}\right)}.\]

\end{cor}

\subsection{Image line charge formula}


We obtained a necessary condition for $|\nabla v|$ to blow up in Section \ref{sec:formulation}. Furthermore, we derived in Theorem \ref{theo:grad} the integral representation of the dominating term. In this section, the singular behavior of $\nabla v$ is described solely by some type of virtual charges of which densities are supported on the part of the real line, namely, $[-c_i,-\alpha]$ and $[\alpha,c_i]$; see Figure \ref{fig:supps} for their illustration.

We define the density functions
\begin{align*}
\ds\varphi_+(s)&=2\alpha\beta e^{2\beta\xi_i}\frac{(s-\alpha)^{\beta-1}}{(s+\alpha)^{\beta+1}}, \quad \psi_+(s)=e^{2\beta\xi_i}\left(\frac{s-\alpha}{s+\alpha}\right)^\beta\quad\mbox{for }\alpha<s<c_i,\quad\mbox{and}\\[2mm]
\ds\varphi_-(s)&=-2\alpha\beta e^{2\beta\xi_i}\frac{(-s-\alpha)^{\beta-1}}{(-s+\alpha)^{\beta+1}},\quad\psi_-(s)=-e^{2\beta\xi_i}\left(\frac{-s-\alpha}{-s+\alpha}\right)^\beta\quad\mbox{for }-c_i<s<-\alpha,
\end{align*}
which slightly varies from those in \cite{imageCharge2018}. Since the proof of the following lemma is the same as in \cite{imageCharge2018}, we state the lemma without proof:

\begin{lemma}{\rm(\cite[Lemma 5.1]{imageCharge2018})}\label{lem:LTandIC}
We simplify the notation as $\Bx=\Bx(\xi,\theta)$ and $\Bs=(s,0)$. Then,
\[\begin{aligned}\ds L(e^{-(2\xi_i-\xi)-i\theta};\beta)&=\int_{\alpha}^{c_i}\ln|\Bx-\mathbf{s}|\varphi_+(s)ds + i\int_{\alpha}^{c_i}\frac{\p}{\p x_2}\ln|\Bx-\mathbf{s}|\psi_+(s)ds+r_+(\Bx)\hfill\mbox{ for $\mathbf{x}\in \Om\backslash \overline{D}$},\\[2mm]
\ds L(e^{-(\xi+2\xi_i)-i\theta};\beta)&=-\int_{-c_i}^{-\alpha}\ln|\Bx-\mathbf{s}|\varphi_-(s)ds - i\int_{-c_i}^{-\alpha}\frac{\p}{\p x_2}\ln|\Bx-\mathbf{s}|\psi_-(s)ds+r_-(\Bx)\hfill\mbox{ for $\Bx\in \Om$,}\end{aligned}\]
where
\[\|\nabla r_+\|_{L^\infty( \Om\backslash \overline{D})}\le1/r_i\quad\mbox{and}\quad\|\nabla r_-\|_{L^\infty(\Om)}\le1/r_i.\]
\end{lemma}

The following asymptotic formulas in terms of image charges correspond to one of the main results of \cite{imageCharge2018}; the proof follows directly from applying Lemma \ref{lem:LTandIC} to Theorem \ref{thm:singularLT} and \ref{theo:grad}.
\begin{theorem}[Image charges for the gradient blow-up term]\label{thm:imgCharge} 
Let $\Om$ and $D$ be given by \eqnref{def:trans:domains} and suppose that $g\in C^{1,\delta}(\Om)$ for some $\delta>1/2$. With the density functions defined in this subsection, the following hold.
\begin{itemize}
\item[\rm(a)] The gradient of the solution $v$ to \eqnref{e0} blows up only in $\Om\backslash\overline{D}$, only if $k\gg1$ and $0<\ep\ll1$.
\item[\rm(b)]
The solution $v$ to \eqnref{e0} satisfies the asymptotic formula
\beq\label{u:aymp}
v(\Bx)=v_*(\Bx)+r(\Bx)\quad\mbox{in }\Om\setminus\overline{D},
\eeq
where $\|\mathbf{e}_\xi\cdot\nabla r\|_{L^\infty(\Om\backslash\overline{D})}$ is bounded independently of $(\ep,k)$ and 
$$v_*(\Bx)=-r_*^2\tau
\left[C_1\int_{\alpha}^{c_i}\ln|\Bx-\Bs|\varphi_+(s)ds+C_2\int_{\alpha}^{c_i}\left(\pd{}{x_2}\ln|\Bx-\Bs|\right)\psi_+(s)ds\right].$$

Alternatively, $v$ admits another asymptotic formula as follows:
$$v(\Bx)=\widetilde{v}_*(\Bx)+\widetilde{r}(\Bx)\quad\mbox{in }\Om\setminus\overline{D},$$ 
where $\|\mathbf{e}_\xi\cdot\nabla \widetilde{r}\|_{L^\infty(\Om\backslash\overline{D})}$ is bounded independently of $(\ep,k)$ and 
$$\widetilde{v_*}(\Bx)=-r_*^2\tau\left[C_1\int_{-c_i}^{-\alpha}\ln|\Bx-\Bs|\varphi_-(s)ds+C_2\int_{-c_i}^{-\alpha}\left(\pd{}{x_2}\ln|\Bx-\Bs|\right)\psi_-(s)ds\right].$$
\end{itemize}
\end{theorem}

\begin{figure}[h!]
\centering
\scalebox{0.8}{
\begin{tikzpicture}[scale=0.6]
\draw[dashed] (5.013490, 0) circle (5);
\draw[dashed] (2.033490, 0) circle (2);

\draw (0.367534, -0.3) node {$\alpha$};
\draw (-0.45, -0.3) node {$-\alpha$};
\draw (2.033490, -0.3) node {$c_i$};
\draw (-2.033490, -0.3) node {$-c_i$};
\draw (2.033490, 1.5) node {$D$};
\draw (5.013490, 4.5) node {$\Om$};

\fill (0.367534, 0) circle (0.07);
\fill (-0.367534, 0) circle (0.07);
\fill (2.033490, 0) circle (0.07);
\fill (-2.033490, 0) circle (0.07);
\draw[ultra thick] (-2.033490, 0) -- (-0.367534, 0);
\draw[ultra thick] (2.033490, 0) -- (0.367534, 0);
\end{tikzpicture}}
$\qquad\qquad$
\scalebox{0.8}{
\begin{tikzpicture}[scale=0.6]
\draw[dashed] (5.013490, 0) circle (5);
\draw[thick] (2.033490, 0) circle (2);

\draw (0.567534, -0.3) node {$\alpha$};
\draw (-2, 0) node {$\theta=0$};
\draw (-2, 2) node {$\theta=\frac{\pi}{2}$};
\draw (-2.1, 3) node {$\theta=\frac{3\pi}{4}$};
\draw (-2, 1) node {$\theta=\frac{\pi}{4}$};
\draw (-2.3, -1) node {$\theta=-\frac{\pi}{4}$};
\draw (-2.3, -2) node {$\theta=-\frac{\pi}{2}$};
\draw (-2.3, -3) node {$\theta=-\frac{3\pi}{4}$};
\draw (6, 0) node {$\theta=\pi$};

\draw (2.033490, 1.5) node {$D$};
\draw (5.013490, 4.5) node {$\Om$};

\fill (0.367534, 0) circle (0.07);
\fill (4.0335, 0) circle (0.07);
\fill (0.2181, -0.8393) circle (0.07);
\fill (0.0664, -0.3615) circle (0.07);
\fill (0.0392, -0.1508) circle (0.07);
\fill (0.2181, 0.8393) circle (0.07);
\fill (0.0664, 0.3615) circle (0.07);
\fill (0.0392, 0.1508) circle (0.07);
\fill (0.0335, 0) circle (0.07);

\draw [->] (5, 0) -- (4.1335, 0);
\draw [->] (-1, -3) -- (0.1181, -0.8393);
\draw [->] (-1, -2) -- (-0.0336, -0.3615);
\draw [->] (-1, -1) -- (-0.0608, -0.1508);
\draw [->] (-1, 3) -- (0.1181, 0.8393);
\draw [->] (-1, 2) -- (-0.0336, 0.3615);
\draw [->] (-1, 1) -- (-0.0608, 0.1508);
\draw [->] (-1, 0) -- (-0.0665, 0);

\end{tikzpicture}}
\caption{\label{fig:supps} The domain is determined by $r_e=5$, $r_i=2$ and $\ep=1/50$. On the left, we describe the supports of image charge density functions in thick line segments. On the right, we draw $\p D$ with $\theta$ coordinates specified, where we will plot in Figure \ref{Dir_blowups} the limit of $\nabla v\cdot\mathbf{e}_\xi$ from outside.}
\end{figure}
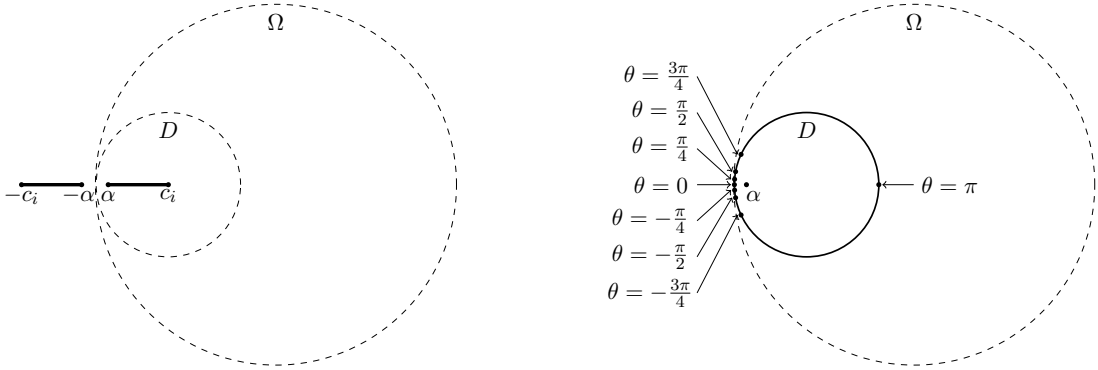

\begin{figure}[h!]
\centering
\includegraphics[width=\textwidth]{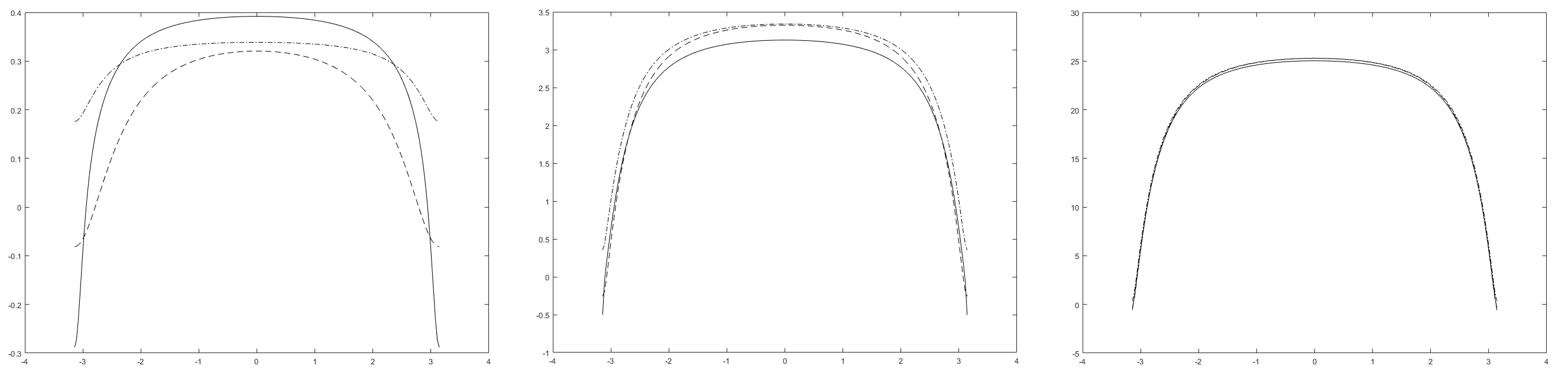}
\phantom{a}\\[1mm]
\includegraphics[width=\textwidth]{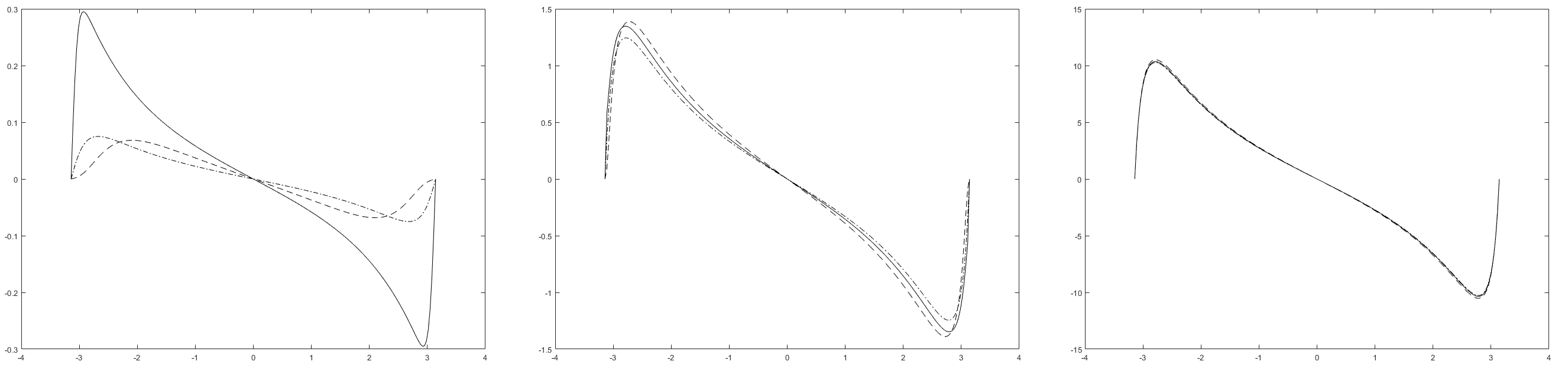}
\caption{\label{Dir_blowups}We plot $\nabla v\cdot\mathbf{e}_\xi$ (filled curves,) $\nabla \widetilde{v}_*\cdot\mathbf{e}_\xi$ (dashed curves) and $\nabla v_*\cdot\mathbf{e}_\xi$ (dashed and dotted curves) with respect to $\theta$ on the circle $\p D$ as illustrated in Figure \ref{fig:supps}. The three figures above assume $g_d(t)=\cos t$ and the three figures below assume $g_d(t)=\sin t$, where $t$ parametrizes $\p D$ by $c_i+r_ie^{it}$. The geometric parameter $\ep=1/(50\cdot 64^{n-1})$, $n=1,2,3$ decreases from the left to the right. For all the cases, $r_e=5$ and $r_i=2$, while $k$ is determined so that $k^2\ep$ is a constant number $2/25$. The result is consistent with Theorem \ref{thm:imgCharge} and Table \ref{tab:CBUR}; one can easily check that $(C_1,C_2)=(1/r_e,0)$ and $(C_1,C_2)=(0,1/r_e)$ for $g_d(t)=\cos t$ and $g_d(t)=\sin t$, respectively.}
\end{figure}

\subsection{Asymptotic formula for $\nabla u$}

Likewise, we may assume
\beq\label{H:linear}H(x_1,x_2)=C_1x_1+C_2x_2.\eeq

\begin{lemma}\label{lemma:u:lseries}
 The solution ${u}$ to \eqnref{e0} with $u=\frac{\p H}{\p\nu}$ for $H$ defined by \eqnref{H:linear} satisfies
\beq\label{eqn:tu:series}
u(\Bx)=C_1\left(x_1+\Re\{\widetilde{F}(\mathbf{x})\}\right)
+C_2\left(x_2+\Im\{\widetilde{F}(\mathbf{x})\}\right),
\eeq
where $\widetilde{F}$ is defined in $\Om$ by
\begin{equation}
\widetilde{F}\left(\mathbf{x}(\xi,\theta)\right)={\rm{const.}}+\begin{cases}
\ds\sum_{n=1}^\infty\left(\widetilde{A}_ne^{n(\xi-i\theta-2\xi_i)}+\widetilde{B}_ne^{n(-\xi-i\theta)}\right)&\mbox{for }\xi_e<\xi<\xi_i,\\[2mm]
\ds\sum_{n=1}^\infty\left(\widetilde{A}_ne^{n(-\xi-i\theta)}+\widetilde{B}_ne^{n(-\xi-i\theta)}\right)&\mbox{for }\xi>\xi_i
\end{cases}\end{equation}
with
\beq \widetilde{A}_n=\frac{-2\alpha(-1)^n\tau}{1+\tau e^{-2n(\xi_i-\xi_e)}}\quadand \widetilde{B}_n=\frac{-2\alpha(-1)^n\tau e^{-2n(\xi_i-\xi_e)}}{1+\tau e^{-2n(\xi_i-\xi_e)}}.\eeq
\end{lemma}

The rest of the procedure for the Neumann-type solution $u$ is analogous to that for the Dirichlet-type solution $v$ in the previous sections. However, the results are completely different, being ``complementary'' to each other in some sense, which is the interesting part. Hence we skip the proof and directly state the results. 

\begin{definition}
\[q(\mathbf{x}(\xi,\theta);\beta):=\begin{cases}
\ds -L(e^{-\xi-2(\xi_i-\xi_e)-i\theta};\beta)-L(e^{\xi-2\xi_i-i\theta};\beta)&\mbox{in }\Om\setminus\overline{D}\\[2mm]
\ds -L(e^{-\xi-2(\xi_i-\xi_e)-i\theta};\beta)-L(e^{-\xi-i\theta};\beta)&\mbox{in }D
\end{cases}\]
\end{definition}

\begin{lemma}\label{lem:repofu:q}
Suppose $g\in C^{1,\delta}(\Om)$ for some $\delta>1/2$ and let $u$ solve \eqnref{e0}.
\begin{itemize}
\item[\rm(a)] For $1<k<\infty$, $\|\nabla {u}\|_{L^\infty(B_e)}$ is bounded independently of $\ep$ and $k$. 

\item[\rm(b)] For $0<k<1$, the solution ${u}$ to \eqnref{e0} satisfies
\beq
{u}(\Bx)= \frac{r_*^2\tau}{2} \big[C_1\Re\{ q(\Bx;\beta)\}+ C_2\Im\{q(\Bx;\beta)\}\big]+r(\Bx),
\eeq
where $\|\nabla r\|_{L^\infty(B_e)}$ is bounded independently of $\ep$ and $k$.
\end{itemize}
\end{lemma}

\begin{theorem}\label{theo:uint:grad} Suppose $g\in C^{1,\delta}(\Om)$ for some $\delta>1/2$. Then, $\nabla u$ blows up only in the limits $\ep\to 0$ and $\tau\to1$, i.e., $0<k\ll1$, whence $\nabla u$ can be approximated as
\beq\begin{aligned}&\ds\nabla u(\mathbf{x}(\xi,\theta))\cdot\mathbf{e}_\xi\\[2mm]
&=\begin{cases}\ds-\frac{r_*\tau}{\sqrt{\ep}}(\cosh\xi+\cos\theta)\left[C_1\Re\left\{P\left(e^{-(\xi+2\xi_i)-i\theta};\beta\right)\right\}+C_2\Im\left\{P\left(e^{-(\xi+2\xi_i)-i\theta};\beta\right)\right\}\right]+O(1)\\[2mm]\hfill\mbox{in }D,\\[2.5mm]
\ds O(1)\quad\mbox{in }\Om\backslash\overline{D},\end{cases}\end{aligned}\eeq
\beq\begin{aligned}&\ds\nabla u(\mathbf{x}(\xi,\theta))\cdot\mathbf{e}_\theta\\[2mm]
&=\begin{cases}\ds\frac{r_*\tau}{\sqrt{\ep}}(\cosh\xi+\cos\theta)\left[C_1\Im\left\{P\left(e^{-(\xi+2\xi_i)-i\theta};\beta\right)\right\}-C_2\Re\left\{P\left(e^{-(\xi+2\xi_i)-i\theta};\beta\right)\right\}\right]+O(1)\quad\mbox{in }D,\\[2.5mm]
\ds\frac{r_*\tau}{\sqrt{\ep}}(\cosh\xi+\cos\theta)\left[C_1\Im\left\{P\left(e^{-(2\xi_i-\xi)-i\theta};\beta\right)\right\}-C_2\Re\left\{P\left(e^{-(2\xi_i-\xi)-i\theta};\beta\right)\right\}\right]+O(1)\\[2mm]\hfill\mbox{in }\Om\backslash\overline{D}.
\end{cases}\end{aligned}\eeq
\end{theorem}

\begin{theorem}[Image charges for the gradient blow-up term]\label{thm:imgCharge:u} 
Suppose that $g\in C^{1,\delta}(\Om)$ for some $\delta>1/2$. Then, the following hold.
\begin{itemize}
\item[\rm(a)] The only cases that the gradient of the solution $u$ to \eqnref{e0} blows up are the following:
\begin{itemize}
\item[(i)] $\|\nabla u\cdot\mathbf{e}_\theta\|_{L^\infty(\Om\backslash\overline{D})}$ blows up only if $0<k,\ep\ll1$.
\item[(ii)] $\|\nabla u\cdot\mathbf{e}_\xi\|_{L^\infty({D})}$ and $\|\nabla u\cdot\mathbf{e}_\theta\|_{L^\infty({D})}$ blow up only if $0<k,\ep\ll1$.
\end{itemize}

\item[\rm(b)]
The solution $u$ to \eqnref{e0} satisfies the asymptotic formula
\beq\label{u:aymp}
u(\Bx)=u_*(\Bx)+r(\Bx)\quad\mbox{in }\Om,
\eeq
where $\|\nabla r\|_{L^\infty(\Om)}$ is bounded independently of $(\ep,k)$ and 

$${u_*}(\Bx)=r_*^2\tau\left[C_1\int_{-c_i}^{-\alpha}\ln|\Bx-\Bs|\varphi_-(s)ds+C_2\int_{-c_i}^{-\alpha}\left(\pd{}{x_2}\ln|\Bx-\Bs|\right)\psi_-(s)ds\right].$$

In addition, $u$ admits an alternative asymptotic formula
\[u(\Bx)=\widetilde{u}_*(\Bx)+\widetilde{r}(\Bx)\quad\mbox{only in }\Om\backslash\overline{D},\]
where $\|\nabla\widetilde{r}\cdot\mathbf{e}_\theta\|_{L^\infty(\Om\backslash\overline{D})}$ is bounded independently of $(\ep,k)$ and
$$\widetilde{u}_*(\Bx)=-r_*^2\tau\left[C_1\int_{\alpha}^{c_i}\ln|\Bx-\Bs|\varphi_+(s)ds+C_2\int_{\alpha}^{c_i}\left(\pd{}{x_2}\ln|\Bx-\Bs|\right)\psi_+(s)ds\right].$$
\end{itemize}
\end{theorem}

Remark that Lemma \ref{lem:approxP} allows the approximation by $\varphi_+$ and $\psi_+$ only in $\Om\setminus\overline{D}$, which is natural from the notion of image charge. The optimal estimates for $\nabla u$ can be derived in the same way; see Table \ref{tab:CBUR}. We end this article by plotting the asymptotic formulas to observe the blow-up phenomena.

\begin{figure}[h!]
\centering
\includegraphics[width=\textwidth]{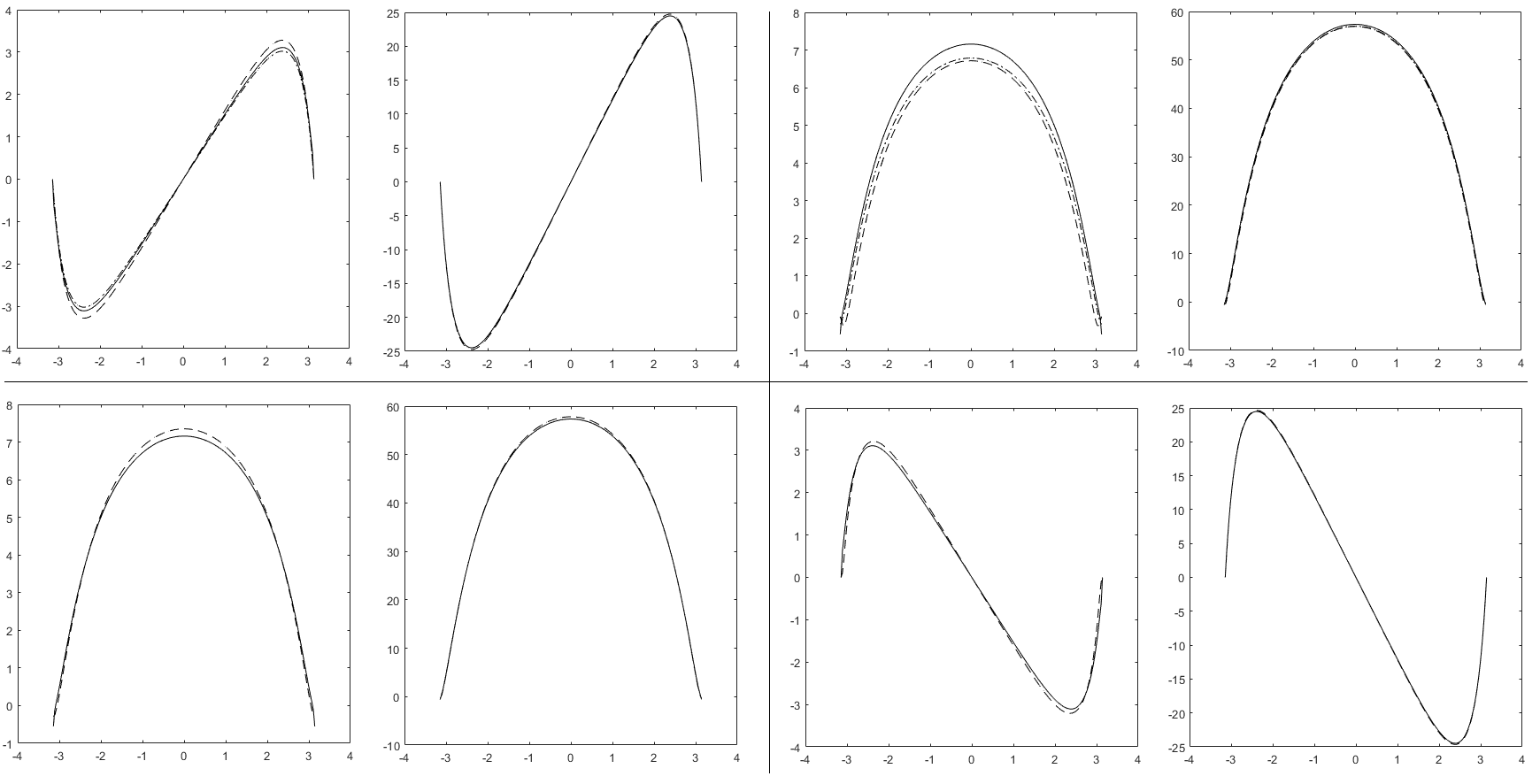}
\caption{\label{Neumann_blowups}We set $r_e=5$, $r_i=2$ and $k^2/\ep=2$ to be constant, but $\ep=1/(50*64^{n})$, where $n=1$ (left) and $n=2$ (right) in each quadrant. The four figures above plot $\nabla u\cdot\mathbf{e}_\theta$ (filled curves,) $\nabla u_*\cdot\mathbf{e}_\theta$ (dashed curves) and $\nabla \widetilde{u}_*\cdot\mathbf{e}_\theta$ (dashed and dotted curves) as in Figure \ref{Dir_blowups}. Note that the gradients in the direction $\mathbf{e}_\theta$ is continuous across $\p D$. The four figures below plot the limit of $\nabla u\cdot\mathbf{e}_\xi$ (filled curves) and $\nabla u_*\cdot\mathbf{e}_\xi$ (dashed curves) from inside. The four figures on the left assume $g(t)=\cos t$ and the others assume $g(t)=\sin t$, where $t$ parametrizes $\p D$ by $c_i+r_ie^{it}$; one can easily check that $(C_1,C_2)=(1/r_e,0)$ and $(C_1,C_2)=(0,1/r_e)$ for $g(t)=\cos t$ and $g(t)=\sin t$, respectively. The result is consistent with Theorem \ref{thm:imgCharge:u} and Table \ref{tab:CBUR}}
\end{figure}

\section{Conclusion}

In this article, the interface problems with a heterogeneity near the boundary were classified as displayed in Table \ref{tab:CBUR} by the asymptotic analysis. Moreover, we found simple asymptotic formulas for both of Dirichlet-type and Neumann-type solutions. The formulas are valid for any boundary conditions satisfying $g\in C^{1,\delta}$ and $g_d\in C^{2,\delta}$ with some $\delta>1/2$. Remark that the approximation for the gradients assumes the error in $O(1)$ as $\ep$ and $k$ vary. It turned out that the Neumann-type and the Dirichlet-type solutions differ significantly in the blow-up feature.

More precisely, we expressed the blow-up term by a linear combination of the two functions, which are independent of the arbitrarily given boundary conditions. The functions are electric potentials generated by the charge densities $\varphi_-$ and $\psi_-$ supported on $[-c_i,-\alpha]$. Only in the intermediate region $\Om\setminus\overline{D}$, which is disjoint from $[-c_i,-\alpha]$, the other two alternative potentials generated by $\varphi_+$ and $\psi_+$ can approximate the gradient blow-up; the potentials are described for several choice of parameters in Figure \ref{Dir_blowups} and Figure \ref{Neumann_blowups}.

We end this article by mentioning some of the possible directions for further analyses of stress concentration in composites. One may take Poisson equations, Lam{\'e} systems, gradient blow-up in perturbed shapes other than disks, etc. into consideration. The results from their mathematical analysis are expected to bring total comprehension of the structure of the solution at once, which can hardly be expected from general numerical tools.

	%

%
\begin{appendix}
\section{Proof of Lemma \ref{lemma:boundedness}}\label{proof:boundedness}
\begin{proof}
Let $\xi_{i,m}$ and $\xi_{e,m+1}$ be defined as in \eqnref{def:xik}. Within this proof, $C$ will denote a constant independent of $\ep$, $\theta$ and $m$. It is sufficient to show that
\beq\label{IandII}\sum_{m=0} ^{\infty} \frac{h(\xi_i,\theta)}{h(\xi_{i,m},\theta)} \big| \nabla \widetilde{H}(\xi_{i,m},\theta)\big|\leq C\eeq
and
\beq\label{IandII2}
\sum_{m=0} ^{\infty} \frac{h(\xi_i,\theta)}{h(\xi_{e,m+1},\theta)} \left| \nabla \widetilde{H}(\xi_{e,m+1},\theta)\right|\leq C.
\eeq 
For notational simplicity we decompose the series in \eqnref{IandII} as

$$\sum_{m=0} ^{\infty} \frac{h(\xi_i,\theta)}{h(\xi_{i,m},\theta)} \big| \nabla \widetilde{H}(\xi_{i,m},\theta)\big| = \sum_{m \le \frac{r_*}{\sqrt{\ep}}} + \sum_{m >\frac{r_*}{\sqrt{\ep}}}:= ~ I + II$$
and separately prove the uniform boundedness.

For $|\theta|\leq\frac{\pi}{2}$, we have from \eqnref{ineq:lem33} and \eqnref{eq:Htild:bigO} that

$$I\leq C\frac{r_*}{\sqrt{\ep}}\Big(\big|\Bx(\xi_{i,m},\theta)\big|+\ep\Big)\leq C\frac{r_*}{\sqrt{\ep}}(\alpha+\ep)=O(1).$$
For $|\theta|>\frac{\pi}{2}$, we deduce from \eqnref{ineq:lem33} that

$$I\leq \frac{C}{|\Bx(\xi_i,\theta)|}\max\left\{\big|\Bx(\xi_{i,m},\theta)\big|:m \le \frac{r_*}{\sqrt{\ep}}\right\}=O(1).$$
Next, we estimate $II$. For $m>\frac{r_*}{\sqrt{\ep}}$, we have $$\xi_{i,m}=\frac{2m}{r_*}\sqrt{\ep}+O(\sqrt{\ep})>1$$ for small $\ep$. Thus, $$\big|\Bx(\xi_{i,m},\theta)\big| < C\alpha$$
and
$$\frac{h(\xi_i,\theta)}{h(\xi_{i,m},\theta)}=\frac{\cosh\xi_i+\cos\theta}{\cosh \xi_{i,m}+\cos\theta}\leq 4e^{-2m(\xi_i-\xi_e)}.$$
Therefore, it follows that
$$
II\leq \sum_{m>\frac{r_*}{\sqrt{\ep}}}4e^{-2m(\xi_i-\xi_e)}C\Big(\big|\Bx(\xi_{i,m},\theta)\big|+\ep\Big)\leq \frac{C(\alpha+\ep)}{\xi_i-\xi_e}= O(1).
$$

One can show \eqnref{IandII2} in the same way.
\end{proof}

\section{Proof of Lemma \ref{thm:singularLT}}
\begin{proof}
The proof is analogous to \cite[Lemma 5.1]{LY:2015:ASCE}. By differentiating Eq. (\ref{e5}), one gets
\[\frac{\partial F}{\partial \xi}=\begin{cases}
\ds\sum_{n=1}^\infty n\left(A_ne^{n(\xi-i\theta-2\xi_i)}-B_ne^{n(-\xi-i\theta)}\right)&\mbox{for }\xi_e<\xi<\xi_i,\\[2mm]
\ds-\sum_{n=1}^\infty n\left(A_ne^{n(-\xi-i\theta)}+B_ne^{n(-\xi-i\theta)}\right)&\mbox{for }\xi>\xi_i,
\end{cases}\]
and
\[\frac{\partial F}{\partial\theta}=\begin{cases}
\ds\sum_{n=1}^\infty(-ni)\left(A_ne^{n(\xi-i\theta-2\xi_i)}+B_ne^{n(-\xi-i\theta)}\right)&\mbox{for }\xi_e<\xi<\xi_i,\\[2mm]
\ds\sum_{n=1}^\infty(-ni)\left(A_ne^{n(-\xi-i\theta)}+B_ne^{n(-\xi-i\theta)}\right)&\mbox{for }\xi>\xi_i.
\end{cases}\]
We expand $A_n$ and $B_n$ defined by \eqnref{def:AnBn} in geometric series:
\[B_n=2\alpha(-1)^n\sum_{m=1}^\infty\tau^me^{mn(-2\xi_i+2\xi_e)}\quad\mbox{and}\quad A_n=-B_ne^{-n(-2\xi_i+2\xi_e)}.\]
Finally, from the relation
\[|z|<1,\ z\in\mathbb{C} \implies\frac{z}{(1-z)^2}=\sum_{n=1}^\infty nz^n \]
we can reduce the double summation into a single power series in $\tau$.

If $-1<\tau<0$ (i.e. $0<k<1$,) we conclude $\nabla F=O(1)$ from \eqnref{ineq:simple:kl1} in Lemma \ref{lem:approxP}. On the other hand, when $0<\tau<1$ (i.e. $k>1$,) \eqnref{ineq:simple:kg1} in Lemma \ref{lem:approxP} gives
\begin{equation}\label{e7}\frac{\partial F}{\partial \xi}=O\left(\frac{1}{h(\xi,\theta)}\right)+\frac{r_*^2\tau}{2}\begin{cases}
\ds\left[P\left(e^{-(2\xi_i-\xi)-i\theta};\beta\right)+P\left(e^{-(2\xi_i-2\xi_e+\xi)-i\theta};\beta\right)\right]&\mbox{for }\xi_e<\xi<\xi_i,\\[2mm]
\ds\left[-P\left(e^{-\xi-i\theta};\beta\right)+P\left(e^{-(2\xi_i-2\xi_e+\xi)-i\theta};\beta\right)\right]&\mbox{for }\xi>\xi_i,
\end{cases}\end{equation}
and
\begin{equation}\label{e8}\frac{\partial F}{\partial\theta}=O\left(\frac{1}{h(\xi,\theta)}\right)-\frac{r_*^2\tau}{2}\begin{cases}
\ds\left[iP\left(e^{-(2\xi_i-\xi)-i\theta};\beta\right)-iP\left(e^{-(2\xi_i-2\xi_e+\xi)-i\theta};\beta\right)\right]&\mbox{for }\xi_e<\xi<\xi_i,\\[2mm]
\ds\left[iP\left(e^{-\xi-i\theta};\beta\right)-iP\left(e^{-(2\xi_i-2\xi_e+\xi)-i\theta};\beta\right)\right]&\mbox{for }\xi>\xi_i.
\end{cases}\end{equation}
Further approximation using \eqnref{e10} shows that $\|\nabla F\|_{L^\infty(B_i)}=O(1)$ and
\[\nabla F(\Bx)=\frac{r_*^2\tau}{2}\nabla q_d(\Bx;\beta)+O(1)\quad \mbox{for all }\Bx\in\Om\setminus\overline{D},\]
which completes the proof.
\end{proof}

\section{Proof of Corollary \ref{prop:optEst}}
\begin{proof}

To obtain the upper bound, a refined version of the inequality \eqnref{e9} is required. Let $s>0$. From the definition of $P$, we have
\begin{equation*}P\left(e^{-s-i\theta};\beta\right)=\int_0^\infty e^{-\beta t}\frac{1+\cosh(s+t)\cos\theta-i\sinh(s+t)\sin\theta}{2(\cosh(s+t)+\cos\theta)^2}d t,\end{equation*}
from which it directly follows that
\beq\label{ineq:P:refined}\left|P\left(e^{-s-i\theta};\beta\right)\right|\le\int_0^\infty \frac{e^{-\beta t}}{2(\cosh(s+t)+\cos\theta)}dt.\eeq
Suppose for the rest of the proof that $\frac{\sqrt{\ep}}{C}\le s\le C\sqrt{\ep}$. Note that for any $t>0$ and $-\pi<\theta\le\pi$,
\[\frac{\cosh s+\cos\theta}{\cosh(s+t)+\cos\theta}\le\frac{\cosh s+1}{\cosh(s+t)+1},\]
which implies
\[\begin{aligned}\ds h(s,\theta)\left|P\left(e^{-s-i\theta};\beta\right)\right|&\le\frac{1}{2\alpha}\int_0^\infty e^{-\beta t-t}\frac{\cosh s+1}{e^{-t}(\cosh(s+t)+1)}dt\\[2mm]
\ds &\le \frac{\cosh s+1}{\alpha e^s}\int_0^\infty e^{-(\beta+1)t}dt=\frac{\cosh s+1}{\alpha(\beta+1)e^{s}}=O\left(\frac{1}{\frac{1}{k}+\sqrt{\ep}}\right),\end{aligned}\]
where the last approximation follows from
\[\beta=\frac{r_*}{4\sqrt{\ep}}\left(\frac{2}{k-1}+O\left(2/({k-1})^2\right)\right)=\frac{r_*}{2k\sqrt{\ep}}+O\left(\frac{1}{k^2\sqrt{\ep}}\right).\]

Next, in order to obtain the optimality, let $\theta=\pi/2$ so that
\[P\left(e^{-s-i\pi/2};\beta\right)=\int_0^\infty e^{-\beta t}\frac{1-i\sinh(s+t)}{2\cosh^2(s+t)}dt.\]
An analogous argument with \cite{imageCharge2018} shows that
\[\Im\left\{P\left(e^{-s-i\pi/2};\beta\right)\right\}\ge\frac{\widetilde{C}}{(\beta+1)(\beta+3)}\quad\mbox{and}\quad\Re\left\{P\left(e^{-s-i\pi/2};\beta\right)\right\}\ge\frac{\widetilde{C}}{\beta+1},\]
which completes the proof.
\end{proof}

\end{appendix}


\end{document}